\documentclass[12pt]{amsart}

\usepackage{amsmath,amsthm,amsfonts,amssymb}

\newtheorem{metatheorem}{metatheorem}[section]
\newtheorem{theorem}[metatheorem]{Theorem}

\newtheorem{lemma}[metatheorem]{Lemma}
\newtheorem{corollary}[metatheorem]{Corollary}

\newtheorem{itdefinition}[metatheorem]{Definition}
\newtheorem{itexample}[metatheorem]{Example}
\newtheorem{itremark}[metatheorem]{Remark}
\newtheorem{itquestion}[metatheorem]{Question}
\newtheorem{itproblem}[metatheorem]{Problem}

    {\begin{itexample}\begin{rm}}{\end{rm}\end{itexample}}
    {\begin{itdefinition}\begin{rm}}{\end{rm}\end{itdefinition}}
\newenvironment{remark}%
    {\begin{itremark}\begin{rm}}{\end{rm}\end{itremark}}
    {\begin{itquestion}\begin{rm}}{\end{rm}\end{itquestion}}
    {\begin{itproblem}\begin{rm}}{\end{rm}\end{itproblem}}

\newcommand{\Xd}{(X, d)}

\newcommand{\M}{\mathcal{M}}

\newcommand{\Mzero}{\M_0}
\newcommand{\Mone}{\M_1}

\newcommand{\MX}{\M(X)}
\newcommand{\MzeroX}{\M_0(X)}
\newcommand{\MoneX}{\M_1(X)}

\newcommand{\Imu}{I(\mu)}

\newcommand{\Imunu}{I(\mu, \nu)}

\newcommand{\N}{\mathbb{N}}
\newcommand{\R}{\mathbb{R}}

\newcommand{\im}{\mathop{\rm im}\nolimits}

\newcommand{\ds}{\displaystyle}

\allowdisplaybreaks

\newcommand{\nwAref}[1]{\ref{#1} of~\cite{NW1}}
\newcommand{\nwBref}[1]{\ref{#1} of~\cite{NW2}}
\newcommand{\nwCref}[1]{\ref{#1} of~\cite{NW3}}

\begin{document}

\title{Finite Quasihypermetric Spaces}

\author{Peter Nickolas}
\address{School of Mathematics and Applied Statistics,
University of Wollongong, Wollongong, NSW 2522, Australia}
\email{peter\_\hspace{0.8pt}nickolas@uow.edu.au}
\author{Reinhard Wolf}
\address{Institut f\"ur Mathematik, Universit\"at Salz\-burg,
Hellbrunnerstrasse~34, A-5020 Salz\-burg, Austria}
\email{Reinhard.Wolf@sbg.ac.at}



\date{}

\thanks{%
The authors are grateful for the financial
support and hospitality of the University of Salzburg and the
Centre for Pure Mathematics in the School of Mathematics and
Applied Statistics at the University of Wollongong.
The authors also wish to thank the referee for a number of detailed
comments, which helped to improve the exposition of the paper.}

\begin{abstract}

Let $\Xd$ be a compact metric space and let
$\MX$ denote the space of all finite signed Borel measures on~$X$.
Define $I \colon \MX \to \R$ by
$\Imu = \int_X \! \int_X d(x,y) \, d\mu(x) d\mu(y)$,
and set $M(X) = \sup \Imu$, where $\mu$ ranges over
the collection of measures in~$\MX$ of total mass~$1$.
The space $\Xd$ is \textit{quasihypermetric} if
$I(\mu) \leq 0$ for all measures~$\mu$ in $\M(X)$
of total mass~$0$ and is \textit{strictly quasihypermetric}
if in addition the equality $I(\mu) = 0$ holds amongst measures~$\mu$
of mass~$0$ only for the zero measure.

This paper explores the constant~$M(X)$ and other geometric aspects
of~$X$ in the case when the space $X$ is finite,
focusing first on the significance of the maximal strictly
quasihypermetric subspaces of a given finite quasihypermetric space
and second on the class of finite metric spaces which are
$L^1$-embeddable. While most of the results are for finite spaces,
several apply also in the general compact case.
The analysis builds upon earlier more general work of the authors
[Peter Nickolas and Reinhard Wolf,
\emph{Distance geometry in quasihypermetric spaces.\ I},
\emph{II} and~\emph{III}].
\end{abstract}

\maketitle

\section{Introduction}
\label{introduction}

Let $\Xd$ be a compact metric space.
We denote by $\MX$ the space of all finite signed Borel measures on~$X$
and by $\MzeroX$ and $\MoneX$, respectively, the subsets of $\MX$
comprising the measures of total mass~$0$ and total mass~$1$.
Define $I \colon \MX \times \MX \to \R$ and $I \colon \MX \to \R$ by
\[
\Imunu = \int_X \! \int_X d(x,y) \, d\mu(x) d\nu(y)
\quad\mbox{ and }\quad
\Imu = I(\mu, \mu)
\]
for $\mu, \nu \in \MX$.
Then we say that $\Xd$ is \textit{quasihypermetric} if $\Imu \leq 0$
for all $\mu \in \MzeroX$ and we say further that 
$\Xd$ is \textit{strictly quasihypermetric}
if in addition $\Imu = 0$ only when $\mu = 0$, for $\mu \in \MzeroX$.
We set
$
M(X) = \sup \Imu,
$
where $\mu$ ranges over~$\MoneX$,
and for $\mu \in \MX$ we define
the function $d_\mu \in C(X)$ by
\[
d_\mu(x) = \int_X d(x, y) \, d\mu(y)
\]
for $x \in X$.

In \cite{NW1} and~\cite{NW2}, we developed a general framework within
which to discuss aspects of the geometry of a compact quasihypermetric
or strictly quasihypermetric space~$X$, with special emphasis on
the behaviour of the geometric constant~$M(X)$. In~\cite{NW3},
we investigated within this framework the role of metric embeddings
in the theory and some of the properties of finite metric spaces.

In the present paper, we explore the case of a finite metric space
in more detail. Specifically, we examine the significance
of the maximal strictly quasihypermetric subspaces
of a given finite quasihypermetric space and
the case of finite $L^1$-embeddable spaces.
We discuss in particular the behaviour of the geometric constant~$M$
in both these contexts. Though most of our results are for
finite spaces, several apply also in the general compact case.

We make free use as necessary of definitions and results from
\cite{NW1, NW2, NW3}, and we reproduce them here only as necessary.
In~\cite{NW1} the background to our work, and in particular related
work by other authors (see \cite{AandS, Bjo, FR1, Hin1, Wol1},
for example), was discussed in some detail,
and this discussion will not be repeated here.

\section{Maximal Strictly Quasihypermetric Subspaces}
\label{maximal}

If $\Xd$ is a compact quasihypermetric metric space
(abbreviated normally to~$X$)
and if $Y$ is a subspace of~$X$ (subspaces will always be assumed
non-empty), then we say that $Y$ is a
\textit{maximal strictly quasihypermetric subspace}
if $Y$ is compact and strictly quasihypermetric and
there is no compact strictly quasihypermetric subspace of~$X$ 
which properly contains~$Y$.

Theorem~\nwAref{qhmconds} gives a number of conditions
which are equivalent to the quasihypermetric
property of a compact metric space~$X$. The definition of the
property is measure-theoretical, but one of the conditions
implies that the property is equivalent to the following condition:
for all $n \in \N$ and for all
$x_1, \ldots, x_n, y_1, \ldots, y_n \in X$,
\[\sum_{i,j=1}^n d(x_i, x_j) + \sum_{i,j=1}^n d(y_i, y_j)
 \leq
2 \sum_{i,j=1}^n d(x_i, y_j).
\]
In particular, $X$ is quasihypermetric if and only if all
finite subspaces of~$X$ are quasihypermetric.
In contrast to this, the strictly quasihypermetric property
in a general compact metric space appears to be
inherently measure-theoretical in character,
and not reducible to a condition on finite subspaces.
This is shown by Theorem~\nwBref{sqhmexample},
which constructs an infinite compact metric space
all of whose proper compact subspaces (and its finite subspaces
in particular) are strictly quasihypermetric
but which is not itself strictly quasihypermetric.

\begin{remark}
\label{remark100}
Quasihypermetric and strictly quasihypermetric spaces
are discussed in~\cite{HLMT} (where they are called
\textit{spaces of negative type} and \textit{strictly
negative type}, respectively). Though the definition
in~\cite{HLMT} corresponds with ours in the quasihypermetric case,
in the strictly quasihypermetric case it corresponds with ours
only in the finite case:
a space is strictly quasihypermetric in the sense of~\cite{HLMT} 
if and only if all its finite subspaces are strictly quasihypermetric
in our sense. (The space in Theorem~\nwBref{sqhmexample} is therefore
of course strictly quasihypermetric in the sense of~\cite{HLMT}.)
\end{remark}

From section~\nwBref{maxinvmeasures}, we recall that
a measure $\mu \in \Mone(X)$ is called \textit{maximal} on~$X$
if $\Imu = M(X)$ and that $\mu \in \MX$ is called
\textit{$d$-invariant} (or simply \textit{invariant}) on~$X$
if there exists $c \in \R$ such that
$d_\mu(x) = c$ for all $x \in X$;
the number~$c$ is then called the \textit{value} of~$\mu$.
The following result from~\cite{NW2}
relates maximal measures and invariant measures.

\begin{theorem}[= Theorem~\nwBref{2.2})]
\label{2.2again}
Let $\Xd$ be a compact metric space.
\begin{enumerate}
\item
If $\mu \in \Mone(X)$ is a maximal measure,
then $\mu$ is $d$-invariant with value $M(X)$.
\item
If $X$ is quasihypermetric and if $\mu \in \Mone(X)$ is $d$-invariant
with value~$c$, then $\mu$ is maximal and $M(X) = c$.
\item
If $X$ is strictly quasihypermetric, then there can exist
at most one maximal measure in~$\Mone(X)$.
\item
If $X$ is strictly quasihypermetric, then there can exist
at most one $d$-invariant measure in~$\MoneX$.
\end{enumerate}
\end{theorem}

We also recall from~\cite{NW2} how the ideas of maximality
and $d$-invariance can be generalized to the case when
maximal and $d$-invariant measures do not exist.
First, from section~\nwAref{chapter2}, we recall that
for any compact quasihypermetric space $\Xd$, the formula
\[
(\mu \mid \nu) : = -\Imunu,
\]
for $\mu, \nu \in \Mzero(X)$,
defines a semi-inner product on $\Mzero(X)$;
it is clear that the semi-inner product is an inner product
if and only if $X$~is strictly quasihypermetric.
The corresponding (semi)norm is of course given by the formula
\[
\| \mu \| : = (\mu \mid \mu)^\frac{1}{2},
\]
for $\mu \in \Mzero(X)$.

Second, from Definition~\nwBref{definitionA},
if ($X, d$) is a compact quasihypermetric space with
$M(X) < \infty$, then a sequence $\mu_n$ in $\Mone(X)$ is
called \textit{maximal} if $I(\mu_n) \to M(X)$ as $n \to \infty$.
Also, by Definition~\nwBref{2.24}, if $\Xd$ is a compact
quasihypermetric space, then a sequence $\mu_n$ in $\Mone(X)$
is called \textit{$d$-invariant with value~$c$}, for $c \in \R$,
if $\| \mu_m - \mu_n \| \to 0$ as $m, n \to \infty$ and
$d_{\mu_n} \to c \cdot \underline{1}$ in $C(X)$ as $n \to \infty$,
where $\underline{1}$ denotes the constant function
$\underline{1}(x) = 1$ for all $x \in X$.
The following result from~\cite{NW2} describes the relation
between maximal sequences and invariant sequences.

\begin{theorem}[= Theorem~\nwBref{theoremD}]
\label{theoremDagain}
Let $\Xd$ be a compact quasihypermetric space.
\begin{enumerate}
\item[(1)]
If $M(X) < \infty$ and $\mu_n$ is a maximal sequence
in $\Mone(X)$, then $\mu_n$~is a $d$-invariant sequence
with value $M(X)$.
\item[(2)]
If $\mu_n$ is a $d$-invariant sequence in $\Mone(X)$ with value~$c$,
then $M(X) = c < \infty$ and $\mu_n$ is a maximal sequence.
\end{enumerate}
\end{theorem}

Theorem~\ref{strict2} and Corollary~\ref{strict2corollary}
below show that a maximal sequence or a maximal measure
on a subspace of a given space may under suitable conditions
also be maximal on the whole space. By Theorem~\ref{theoremDagain},
we may of course express these results equivalently in terms of
$d$-invariant sequences of measures of mass~$1$ or
$d$-invariant measures of mass~$1$.
The following result shows that a $d$-invariant measure of mass~$0$
on a compact subspace of a compact quasihypermetric space
must always be $d$-invariant on the whole space.

\begin{theorem}
\label{strict1}
Let $\Xd$ be a compact quasihypermetric space
and $Y \subseteq X$ a compact subspace.
If $\mu \in \Mzero(Y)$ is $d$-invariant on~$Y$ with value~$c$,
then $\mu$ is also $d$-invariant on~$X$ with value~$c$.
\end{theorem}

\begin{proof}
Since $\mu \in \Mzero(Y)$ and $d_\mu$ is constant on~$Y$,
we have $I(\mu) = 0$, where we regard $I$ as a functional on~$\M(Y)$.
Hence, regarding~$\mu$ in the obvious way as a measure on~$X$,
we also have $I(\mu) = 0$, regarding~$I$ now as a functional on~$\MX$.
But Lemma~\nwAref{2.7} now implies
that $d_\mu$~is constant on~$X$, and since its value on~$Y$ is~$c$,
its value on~$X$ must also be~$c$. 
\end{proof}

\begin{lemma}
\label{RW1}
Let $\Xd$ be a compact quasihypermetric space
and let $Y \subseteq X$ be a \textup{(}compact\textup{)}
maximal strictly quasihypermetric subspace of~$X$.
Then for all $x \in X \setminus Y$ there exists a unique measure
$\mu_x \in \Mone(Y)$ such that $I(\mu_x - \delta_x) = 0$,
where $\delta_x$ denotes the point measure at~$x$.
\end{lemma}

\begin{proof}
Fix $x \in X \setminus Y$. By the maximality of~$Y$,
the subspace $Y \cup \{x\}$ is non-strictly quasihypermetric,
so there exists $\nu \in \Mzero(Y \cup \{x\})$ such that
$I(\nu) = 0$ and $\nu \neq 0$.
Write $\nu = \phi + \alpha \delta_x$
for some $\phi \in \M(Y)$ and some $\alpha \in \R$ such that
$\phi(Y) + \alpha = 0$.

If $\alpha = 0$, we have $\nu = \phi \in \Mzero(Y)$ and $I(\nu) = 0$,
which implies that $\nu = 0$ since $Y$ is strictly quasihypermetric,
and this is a contradiction.
Therefore $\alpha \neq 0$, and clearly $\mu_x = -(1/\alpha) \phi$
satisfies $I(\mu_x - \delta_x) = 0$.

Suppose that there exist $\mu_1, \mu_2 \in \Mone(Y)$
such that $I(\mu_1 - \delta_x) = I(\mu_2 - \delta_x) = 0$.
Then by part~(5) of Lemma~\nwAref{2.7} we find that
$d_{\mu_1-\delta_x}$ and $d_{\mu_2-\delta_x}$ are both
constant functions on~$X$. Hence there exists $c \in \R$ such that
$d_{\mu_1-\mu_2} = c \cdot \underline{1}$.
It follows that $I(\mu_1-\mu_2) = 0$,
and since $Y$ is strictly quasihypermetric, we have $\mu_1 = \mu_2$.
\end{proof}

\begin{theorem}
\label{RW2}
Let $\Xd$ be a compact quasihypermetric space and suppose that $Y$
is a finite maximal strictly quasihypermetric subspace of~$X$.
Then $X$ is finite. 
\end{theorem}

\begin{proof}
Let $Y = \{ y_1, \dots, y_n \}$. By Lemma~\ref{RW1}, we have
for each $x \in X \setminus Y$ a measure $\mu_x \in \Mone(Y)$
such that $I(\mu_x - \delta_x) = 0$. By part~(5) of Lemma~\nwAref{2.7},
there exist constants~$c_x$ such that
$d_{\mu_x-\delta_x} = c_x \cdot \underline{1}$ on~$X$ for each~$x$.
Writing $i(x) = d_{\delta_x}$ for $x \in X$, we have
$i(x) = d_{\mu_x} - c_x \cdot \underline{1}$ on~$X$.
Since $\mu_x$ is supported on~$Y$, we have
\[
i(x) \in [ i(y_1), \dots, i(y_n), \underline{1} ],
\]
where the brackets $[ \ldots ]$ denote the linear span.
Recall (see section~\nwAref{chapter1})
that $T \colon \MX \to C(X)$ is the linear map defined
by setting $T(\mu) = d_\mu$ for $\mu\in \MX$.
Then by part~(2) of Lemma~\nwAref{newlemma}, we have
\[
\im T \subseteq [ i(y_1), \dots, i(y_n), \underline{1} ].
\]
Now Theorem~\nwAref{newtheorem} implies that $X$ is finite,
as required.
\end{proof}

\begin{theorem}
\label{strict2}
Let $\Xd$ be a compact quasihypermetric space with $M(X) < \infty$,
and suppose that $Y \subseteq X$ is a \textup{(}compact\textup{)}
maximal strictly quasihypermetric subspace of~$X$. Then each
maximal sequence $\mu_n$ on the subspace~$Y$, when considered as
a sequence in $\Mone(Y)$, is also a maximal sequence on the space~$X$,
when considered as a sequence in $\Mone(X)$. Also, $M(Y) = M(X)$.
\end{theorem}

\begin{proof}
Let $\mu_n$ be a maximal sequence on~$Y$ and consider any
$x \in X \setminus Y$. Write $Z = Y \cup \{x\}$.
Since $Y$ is a maximal strictly quasihypermetric subspace of~$X$,
the subspace~$Z$ is quasihypermetric but not strictly quasihypermetric,
so there exists a non-zero measure $\nu \in \Mzero(Z)$
such that $I(\nu) = 0$.
By Lemma~\nwAref{2.7}, $\nu$ is $d$-invariant on~$Z$
and by Theorem~\nwAref{2.9.5} it follows that $d_\nu = 0$ on~$Z$,
and so, by Theorem~\ref{strict1}, we have $d_\nu = 0$ on~$X$.
Observe that we must have $\nu(\{x\}) \neq 0$, since $\nu \neq 0$ and
$Y$ is strictly quasihypermetric.

Since $M(Z) \leq M(X) < \infty$, there exists a maximal sequence
$\phi_n$ in $\Mone(Z)$ for the subspace~$Z$.
Hence for all $n \in \N$ we can define
\[
\psi_n = \phi_n - \lambda_n \nu,
\]
where
\[
\lambda_n = \frac{\phi_n \bigl(\{x\}\bigr)}{\nu\bigl(\{x\}\bigr)}.
\]
Since $d_\nu = 0$ on~$X$, we obtain
$I(\psi_m -\psi_n) = I(\phi_m - \phi_n)$ for all $m, n \in \N$
and $d_{\psi_n} = d_{\phi_n}$ on~$X$.
Now $\phi_n$ is a maximal sequence on~$Z$, and hence a $d$-invariant
sequence, by Theorem~\ref{theoremDagain}, and so
$\| \psi_m -\psi_n \| = \| \phi_m - \phi_n \| \to 0$
as $m, n \to \infty$ and $d_{\psi_n} = d_{\phi_n} \to M(Z)$
uniformly on~$Z$ and hence on~$Y$.
Noting that $\psi_n(\{x\}) = 0$ and applying Theorem~\ref{theoremDagain} 
to the sequence~$\psi_n$ in $\Mone(Y)$, we find that
$\psi_n$ is a maximal sequence on~$Y$ and that $M(Y) = M(Z)$.

The sequence $\mu_1, \psi_1, \mu_2, \psi_2, \dots$ is of course
a maximal sequence on~$Y$ and therefore $d$-invariant,
and hence, regarding $\mu_n$ and~$\psi_n$ as measures on~$X$,
we have $\| \mu_n -\psi_n \| \to 0$ as $n \to \infty$.
By Theorem~\nwAref{2.10}, there is a constant $c \geq 0$ such that
$\| d_{\mu_n -\psi_n} \|_\infty \leq c \| \mu_n - \psi_n \|$,
and so $d_{\mu_n} - d_{\psi_n} \to 0$ uniformly on~$X$.
Since $\phi_n$ is a maximal and hence $d$-invariant sequence
on~$Z$, it follows that
$d_{\psi_n}(x) = d_{\phi_n}(x) \to M(Z) = M(Y)$ as $n \to \infty$.
But $d_{\mu_n}(x) - d_{\psi_n}(x) \to 0$ as $n \to \infty$,
and therefore $d_{\mu_n}(x) \to M(Y)$ as $n \to \infty$.

Since $x \in X \setminus Y$ was arbitrary and
$\mu_n$ is a maximal and therefore $d$-invariant sequence on~$Y$,
we have $d_{\mu_n} \to M(Y) \cdot \underline{1}$ uniformly on~$Y$,
and we find that $d_{\mu_n} \to M(Y) \cdot \underline{1}$
pointwise on~$X$. Applying Theorem~\nwAref{2.10} again, we have
\[
\| d_{\mu_m} - d_{\mu_n} \|_\infty
 =
\| d_{\mu_m -\mu_n} \|_\infty
 \leq
c \| \mu_m - \mu_n \|
 \to
0
\]
as $m, n \to \infty$, and hence there exists $f \in C(X)$
such that $d_{\mu_n} \to f$ uniformly on~$X$.
Therefore, $d_{\mu_n} \to M(Y) \cdot \underline{1}$ uniformly on~$X$
and so, since $\| \mu_n - \mu_m \| \to 0$ as $n, m \to \infty$,
we conclude that $\mu_n$ is a $d$-invariant and hence maximal sequence
on~$X$. It follows immediately that $M(Y) = M(X)$.
\end{proof}

A maximal measure $\mu \in \Mone(Y)$ of course defines
a degenerate maximal sequence $\mu, \mu, \dots$ on~$Y$,
and since maximal and $d$-invariant measures coincide on
quasihypermetric spaces by Theorem~\ref{2.2again},
we obtain the following corollary.

\begin{corollary}
\label{strict2corollary}
Let $X$ be a compact quasihypermetric space such that $M(X) < \infty$
and let $Y$ be a \textup{(}compact\textup{)} maximal
strictly quasihypermetric subspace of~$X$. Then we have the following.
\begin{enumerate}
\item[(1)]
If $\mu \in \Mone(Y)$ is a maximal measure on~$Y$, then $\mu$ is
a maximal measure on~$X$.
\item[(2)]
If $\mu \in \Mone(Y)$ is such that $d_\mu$ is constant on~$Y$,
then $d_\mu$ is also constant on~$X$, and $M(Y) = M(X)$.
\end{enumerate}
\end{corollary}

\begin{theorem}
\label{RW3}
Let $(X = \{x_1, \dots, x_n\}, d)$ with $n > 1$
be a finite quasihypermetric space and let $Y \subseteq X$
be a maximal strictly quasihypermetric subspace of~$X$.
Further, let $D = \bigl( d(x_i, x_j) \bigr)_{i,j=1}^n$
be the distance matrix of~$X$ and let $r$ be the rank of~$D$.
Then
\[
|Y|
 =
\begin{cases}
  r,   & \mbox{if } M(X) < \infty, \\
  r-1, & \mbox{if } M(X) = \infty.
\end{cases}
\]
\end{theorem}

\begin{proof}
Write $Y = \{y_1, \dots, y_s\}$.
As in the proof of Theorem~\ref{RW2}
(and using the same notation), we have
$\im T \subseteq [ i(y_1), \dots, i(y_s), \underline{1}]$.
But $\underline{1} \in \im T$, by Theorem~\nwCref{alg1},
and of course $[ i(y_1), \dots, i(y_s) ] \subseteq \im T$,
so we have
\[
\im T = [ i(y_1), \dots, i(y_s), \underline{1} ].
\]

Since $Y$ is strictly quasihypermetric, we know by Theorem~\nwAref{2.19}
that the operator $T_Y \colon \M(Y) \to C(Y)$ defined by
$T_Y(\mu) = d_\mu$ for $\mu \in \M(Y)$ is injective,
and since the set $\{ \delta_{y_1}, \dots, \delta_{y_s} \}$
is clearly a basis for $\M(Y)$, it follows that
\[
\dim\, [ i(y_1), \dots, i(y_s) ] = s.
\]

Now we claim that $\underline{1} \in [ i(y_1), \dots, i(y_s) ]$
if and only if $M(X) < \infty$.
Suppose first that $\underline{1} \in [ i(y_1), \dots, i(y_s) ]$.
Then there are $\alpha_1, \dots, \alpha_s \in \R$ such that
$d_\mu = \underline{1}$ on~$X$, where 
$\mu = \alpha_1 \delta_{y_1} + \dots + \alpha_s \delta_{y_s} \in \M(Y)$.
It follows that $I(\mu) = \mu(Y)$.
Now if $\mu(Y) = 0$, then the fact that $Y$ is strictly quasihypermetric
would imply that $\mu = 0$, a contradiction, so we have $\mu(Y) \neq 0$.
Hence the measure $\mu' = (1/\mu(Y)) \mu \in \Mone(X)$
is $d$-invariant on~$X$ and therefore maximal on~$X$,
by Theorem~\nwBref{2.2}, and it follows that $M(X) < \infty$.

Second, suppose that $M(X) < \infty$.
Then of course $M(Y) < \infty$, and there exists a unique
maximal measure $\mu \in \Mone(Y)$ on~$Y$,
by Theorems \ref{theoremDcorollary} and~\nwBref{2.2}.
Now Theorem~\ref{2.2again}, part~(1) and Corollary~\ref{strict2corollary},
part~(2) imply that $d_\mu = M(X) \cdot \underline{1}$
on~$X$, and we must have $M(X) > 0$ since $|X| = n > 1$.
Therefore, we have $\underline{1} \in [ i(y_1), \dots, i(y_s) ]$.

From the above, we conclude that
\[
\dim\, [ i(y_1), \dots, i(y_s), \underline{1} ]
 =
\begin{cases}
  s,   & \mbox{if } M(X) < \infty, \\
  s+1, & \mbox{if } M(X) = \infty.
\end{cases}
\]
Finally, since $D$ is the matrix of the operator~$T$, we have
$r = \dim(\im T)$, and the result follows.
\end{proof}

The following result is immediate (see also Remark~\ref{Knr25} below).

\begin{corollary}
\label{RW4}
If $\Xd$ is a finite quasihypermetric space, then all maximal
strictly quasihypermetric subspaces of~$X$ have the same cardinality.
\end{corollary}

\begin{theorem}
\label{danish_thm}
Let $\Xd$ be a finite metric space with $|X| > 1$ and distance matrix~$D$.
Then $X$ is strictly quasihypermetric if and only if
$X$ is quasihypermetric, $M(X) < \infty$ and $D$ is non-singular.
\end{theorem}

\begin{proof}
For the forward implication, we observe that if
$X$~is strictly quasihypermetric, then $M(X) < \infty$
by Theorem~\nwBref{2.13} and (since $|X| > 1$)
$D$ is non-singular by Remark~\nwCref{algremark2}.
The reverse implication is immediate from Theorem~\nwAref{2.20}.
\end{proof}

We recall the following definition, due to Kelly~\cite{Kel2}.
Let $\Xd$ be a metric space. If for all $n \in \N$
and for all $a_1, \ldots, a_n, b_1, \ldots, b_{n+1} \in X$ we have
\[
\sum_{i=1}^n \sum_{j=1}^n d(a_i, a_j)
    + \sum_{i=1}^{n + 1} \sum_{j=1}^{n+1} d(b_i, b_j)
\leq
2 \sum_{i=1}^n \sum_{j=1}^{n+1} d(a_i, b_j),
\]
then $\Xd$ is said to be a \textit{hypermetric space}.
Kelly shows that the hypermetric property implies the
quasihypermetric property.
Using Lemma~1.2 of~\cite{Ass1} and Theorem~\nwCref{sphere400},
we see that finite hypermetric spaces have $M$ finite,
and we therefore have the following result from~\cite{HLMT}.

\begin{corollary}[Theorem~5.2 of \cite{HLMT}]
\label{danish_cor}
If the finite metric space $\Xd$ is hypermetric and has non-singular
distance matrix, then $\Xd$ is strictly quasihypermetric.
\end{corollary}

\begin{remark}
\label{danish_rem}
By Theorem~\nwAref{2.6},
if $\Xd$ is a compact metric space with $M(X) < \infty$,
then $X$ is quasihypermetric, so that the quasihypermetric
condition could be omitted from the statement of Theorem~\ref{danish_thm}.
We observe, however, that neither of the other conditions
can be omitted.

First, Theorem~\nwBref{2.9}, with Remark~\nwCref{algremark3},
gives a $5$-point quasihypermetric, non-strictly quasihypermetric
space with $M$ infinite and non-singular distance matrix.
Second, a $4$-point space formed by choosing two diametrically
opposite pairs of points from a circle equipped with the
arc-length metric is quasihypermetric and has $M$ finite
but is not strictly quasihypermetric
(see Example~\nwBref{4ptexample}).
\end{remark}

\section{The Value of $M$ in Finite Quasihypermetric Spaces}
\label{Knr}

In this section, we discuss the numerical value of~$M$ in a finite
quasihypermetric space, considered as a function of the cardinality
of the space and the cardinality of its maximal
strictly quasihypermetric subspaces (see Corollary~\ref{RW4}).

We will say that a finite metric space $\Xd$ is \textit{of type $(n, r)$},
for positive integers $n$ and~$r$, if $|X| = n$, $X$ has metric diameter
$D(X) = 1$, $X$~is quasihypermetric, $M(X) < \infty$ and the
maximal strictly quasihypermetric subspaces of~$X$ have cardinality~$r$.
For pairs $(n, r)$ for which the collection of spaces of type $(n, r)$
is non-empty, we define
\[
K(n, r) = \sup \{ M(X) : X \mbox{ has type } (n, r) \}.
\]
The main results of this section  give information about $K(n, r)$.

First, we briefly recall some definitions and results from~\cite{NW3}.
We noted in section~\nwCref{embeddingsI}
(and earlier in section~\nwAref{sect:qhm})
the result of Schoenberg~\cite{Sch3} that a separable metric space
$\Xd$ is quasihypermetric if and only if the metric space
$(X, d^\frac{1}{2})$ can be embedded isometrically
in the Hilbert space~$\ell^2$.
In particular, if $X$ is a finite space,
then $\Xd$ is quasihypermetric if and only if $(X, d^\frac{1}{2})$
can be embedded isometrically in a euclidean space
of suitable dimension. Following~\cite{NW3}, we refer to an embedding
of $(X, d^\frac{1}{2})$ in a euclidean space or in Hilbert space
as a \textit{Schoenberg-embedding} or, for short,
an \textit{S-embedding} of~$X$.

Consider a set $\{p_1, \dots, p_n\}$ of points in a euclidean space
and suppose that these points are in the obvious way
the points of an S-embedding of a (quasihypermetric) metric space
$(X = \{x_1, \dots, x_n\}, d)$,
so that $d(x_i, x_j) = \|p_i - p_j\|^2$ for all $i$ and~$j$.
Fix any three distinct points from $\{p_1, \dots, p_n\}$.
Then by applying the cosine rule in the triangle defined by
the three points and applying the triangle inequality to the three
corresponding points of~$X$, we see that the three angles formed
by the three points are less than or equal to $\pi/2$.
Conversely, it is straightforward to check that if $p_1, \dots, p_n$
are (distinct) points in a euclidean space which satisfy
the angle condition just noted, then the function~$d$
defined by $d(x_i, x_j) = \|p_i - p_j\|^2$ for all $i$ and~$j$
is a metric on $X = \{x_1, \dots, x_n\}$ (necessarily satisfying
the quasihypermetric property).

We will refer to a (finite) configuration of points
in a euclidean space satisfying the above angle condition
as a \textit{non-obtuse} configuration,
and to one in which all angles are strictly less than $\pi/2$
as an \textit{acute} configuration.

We require below one further aspect of the correspondence just
described between finite quasihypermetric spaces and finite
non-obtuse confugurations of points in euclidean spaces.
By part~(3) of Theorem~\nwCref{sphere400}, the metric space
in such a correspondence is strictly quasihypermetric if and
only if the configuration is affinely independent.

\begin{theorem}
\label{D&G}
Let $\Xd$ with $|X| = n$ be a quasihypermetric space,
and let $Y$ with $|Y| = r$ be a maximal strictly quasihypermetric
subspace of~$X$. Then $n \leq 2^{r-1}$.
\end{theorem}

\begin{proof}
Consider an S-embedding of~$X$ into a euclidean space of any dimension,
writing $X'$ for the image of~$X$ and $Y'$ for the image of~$Y$.
By Theorem~\nwCref{sphere400}, as just noted,
$Y'$ is an affinely independent set,
and for each $x' \in X' \setminus Y'$,
the set $Y' \cup \{x'\}$ is affinely dependent. It follows
that $X' \setminus Y'$ lies in the affine hull of~$Y'$.
But the affine hull of~$Y'$ is an $(r-1)$-dimensional affine subspace
of the original euclidean space, so we may take the original space
to be $\R^{r-1}$.
 
As noted above, any S-embedding of~$X$ in a euclidean space
forms a non-obtuse configuration.
But a result of Danzer and Gr\"{un}baum~\cite{D&G} shows that
at most $2^{r-1}$ points can be placed in $\R^{r-1}$ so as to form
such a configuration, and we conclude that $n \leq 2^{r-1}$.
\end{proof}

Thus, for a given~$r$, the range of possible values of~$n$
is given by $r \leq n \leq 2^{r-1}$. We claim that all
values of~$n$ in this range can be realized.

Given $a_1, \dots, a_{r-1} > 0$, consider in $\R^{r-1}$
the set $S$ consisting of the $2^{r-1}$~corners
\[
(\pm a_1, \dots, \pm a_{r-1})
\]
of a rectangular parallelepiped, and fix an affinely independent
subset~$A$ of~$S$ with $r$ elements, such as $(a_1, \dots, a_{r-1})$
together with the $r-1$ points
\[
(-a_1, a_2, \dots, a_{r-1}),
(a_1, -a_2, \dots, a_{r-1}), \dots,
(a_1, a_2, \dots, -a_{r-1}).
\]
Clearly, $S$ is a non-obtuse configuration,
and so by our comments above is the S-embedding
of a quasihypermetric metric space~$X$ with $2^{r-1}$ elements.
Further, by Theorem~\nwCref{sphere400},
$A$ is the S-embedding of a maximal strictly quasihypermetric
subspace~$Y$ of~$X$. Now, by removing the elements of $X \setminus Y$
one by one from~$X$, we obtain for each~$n$ in the range
$r \leq n \leq 2^{r-1}$ a quasihypermetric space with $n$ elements
and a maximal strictly quasihypermetric subspace with $r$ elements,
proving our claim. Moreover, $S$ lies on the sphere with centre~$0$
and radius $(a_1^2 + \dots + a_{r-1}^2)^{1/2}$,
and so by Theorem~\nwCref{sphere400}, $X$~and all the subspaces
just constructed have $M$ finite.
We have the following immediately.

\begin{corollary}
For all $n > 1$, the range of~$r$ for which $K(n, r)$ is defined
is given by $\lceil \log_2 n \rceil + 1 \leq r \leq n$.
\end{corollary}

Note that the value $n = 1$ does not arise here, since attention
is restricted in the definition of~$K$ to spaces of diameter~$1$.

The following simple observation will be useful.

\begin{lemma}
\label{simple}
Suppose that $\Xd$ is a finite metric space which has an S-embedding
into a sphere in some euclidean space and that the points of the
embedded set include two diametrically opposite points of the sphere.
Then $M(X)/D(X) = 1/2$.
\end{lemma}

\begin{proof}
We may assume that the S-embedding is into a euclidean space
of minimal dimension, since otherwise we may pass to an affine
flat of the lowest dimension which contains the S-embedded set,
and the existence of the diametrically opposite pair
ensures that the flat intersects the original sphere
in a sphere of lower dimension with the same radius.
Then Theorem~\nwCref{sphere200} implies that $M(X) = 2 r^2$,
where $r$ is the radius of the sphere. But from the definition
of an S-embedding, the metric diameter of~$X$ is
$D(X) = (2r)^2 = 4 r^2$, so the result follows.
\end{proof}

\begin{theorem}
\label{Kbox}
$K(2^{r-1}, r) = 1/2$ for all $r \geq 2$.
\end{theorem}

\begin{proof}
Danzer and Gr\"{un}baum~\cite{D&G} show not only that the largest
non-obtuse configuration in $\R^{r-1}$ has $2^{r-1}$ points,
but that every such configuration consists of the vertices
of a rectangular parallelepiped. Application of the lemma
therefore gives the result immediately.
\end{proof}

\begin{lemma}
\label{Knr5}
Let $S$ be a non-degenerate $(k-1)$-sphere in~$\R^k$
and let $P$ be a point of~$S$.
If $A$ is an affine flat of dimension less than or equal to~$k-1$
passing through~$P$, then every open ball centred at~$P$
contains a point of $S \setminus A$.
\end{lemma}

\begin{proof}
Given an open ball at~$P$, choose a $(k-1)$-sphere~$S_1$
centred at~$P$ which lies within the ball and has radius
less than the diameter of~$S$.
Then $S_2 = S \cap S_1$ is a $(k-2)$-sphere, and therefore
contains a maximal affinely independent subset with $k$ points
which spans an affine flat~$B$ of dimension $k-1$.
Clearly, $P$ does not lie in~$B$, so $A$ cannot contain~$S_2$,
since if it did it would contain an affinely independent
set of $k+1$ points, contradicting our assumption.
Any point chosen from $S_2 \setminus A$ is therefore as required.
\end{proof}

\begin{theorem}
\label{Knr10}
$K(n, r) = \infty$ for $n \geq 5$ and
$\ds\Bigl\lceil \frac{n+5}{2} \Bigr\rceil \leq r \leq n$.
\end{theorem}

\begin{proof}
Fix $n$ and~$r$ as above, and note that the
second condition is equivalent to the condition $r \leq n \leq 2r-5$. 
In~\cite{D&G}, a construction is given, for any $d \in \N$,
of an acute configuration of $2d-1$ points in~$\R^d$.
Set $d = r-2$, and note that $d \geq 3$ and $d+2 \leq n \leq 2d-1$.
Then we can choose in~$\R^d$ an $n$-point acute configuration
$P = \{p_1, \dots, p_n\}$ which, after the application of suitable
similarity transformations, we may assume to have metric diameter
$D(P) = 1$
and to lie in the closed ball of radius~$1$ centred at the origin
(a smaller ball can be chosen for the given diameter
by Jung's theorem, but minimality is not of concern here).

Embed $\R^d$ in $\R^{d+1}$ as the hyperplane with $(d+1)$th
coordinate~$0$.
For $\rho > 0$, let $c_0$ be the point $(0, \dots, 0, \rho) \in \R^{d+1}$. 
It is easy to check that for each $p \in P$ we have
$\rho \leq \|c_0 - p\| \leq \rho + 1/(2\rho)$.
Therefore, there exist distinct points $p'_1, \dots, p'_n$
which lie on the sphere~$S$ with centre~$c_0$ and radius~$\rho$
and satisfy $\|p_i - p'_i\| \leq 1/(2\rho)$ for all~$i$.
Since the angles among three points are continuous functions
of the points, a large enough choice of~$\rho$ will ensure
both that $P' = \{p'_1, \dots, p'_n\}$
is an acute configuration and that the metric diameter $D(P')$ of~$P'$
is as close as we wish to~$1$.

Denote by~$m$ the cardinality of any (and hence every) maximal
affinely independent subset of~$P'$. Since $P' \subseteq \R^{d+1}$,
we have $m \leq d+2$.

If $m$ is strictly less than~$d+2$,
then the affine flat spanned by~$P'$ is of dimension at most~$d$,
and we can apply Lemma~\ref{Knr5}, with $k$ in the lemma equal to~$d+1$, 
replacing one chosen point of~$P'$ by another point of~$\R^{d+1}$
that is arbitrarily close, is outside the affine flat,
and still lies on the sphere~$S$.
This yields an $n$-point configuration which lies on~$S$
and whose maximal affinely independent subsets have cardinality~$m+1$.

Therefore, if, beginning with~$P'$, we carry out such a replacement
process $d + 2 - m \geq 0$ times, we obtain an $n$-point configuration~$Q$
which lies on~$S$ and whose maximal affinely
independent subsets have the largest possible cardinality,
namely,~$d+2$. Also, by suitable choice of~$\rho$
and of the balls used in the applications of Lemma~\ref{Knr5},
we can ensure that $Q$ is an acute configuration and that
the metric diameter~$D(Q)$ of~$Q$ is arbitrarily close to~$1$.

Now consider the quasihypermetric space~$X_Q$
corresponding to~$Q$ (see our remarks at the start of this section).
This space has cardinality~$n$ and (again by our remarks above)
has maximal strictly quasihypermetric subspaces of cardinality~$d+2$.
Also, $M(X_Q) = 2\rho^2$, by Theorem~\nwCref{sphere200},
and so, since $\rho$ may be taken arbitrarily large, we have
$K(n, d+2) = K(n, r) = \infty$, as required.
\end{proof}

\begin{corollary}
\label{Knr20}
\mbox{}
\begin{enumerate}
\item[(1)]
$K(n, n) = \infty$ for $n \geq 5$.
\vspace{4pt}
\item[(2)]
$K(n, n-1) = \infty$ for $n \geq 7$.
\end{enumerate}
\end{corollary}

Though the above statements follow immediately from the theorem,
we think it is worthwhile to outline direct proofs
by the construction of quite concrete spaces.

\begin{theorem}[= Corollary \ref{Knr20}, part (1)]
\label{Knr20.1}
If $n \geq 5$, then for every $K > 0$, there exists
a strictly quasihypermetric space $(Z, d)$
with $|Z| = n$ and $D(Z) = 1$ such that $M(Z) > K$.
\end{theorem}

\begin{proof}
Write $m = n-2$, so that $m \geq 3$.
Give $X = \{x_1, \dots, x_m\}$ the discrete metric~$d_1$
and give $Y = \{y_1, y_2\}$ the discrete metric~$d_2$.
Let $Z = X \cup Y$, set $c = (m-1)/(2m) + 1/4 + \epsilon$
for any $\epsilon$ satisfying $0 < \epsilon \leq (m+2)/(4m)$,
and define $d \colon Z \times Z \to \R$ as in
Theorem~\nwBref{2.1}. Then the latter theorem shows that $(Z, d)$
is strictly quasihypermetric and that $D(Z) = 1$,
and Theorem~\nwBref{2.1.4} gives a measure of mass~$1$
on~$Z$ which is invariant with value
\[
\frac{1}{32\epsilon} \biggl(\frac{m-2}{m}\biggr)^2
  + \frac{m-1}{2m} + \frac{1}{4} + \frac{\epsilon}{2}.
\]
The result follows.
\end{proof}

\begin{theorem}[= Corollary \ref{Knr20}, part (2)]
\label{Knr20.2}
If $n \geq 7$, then for every $K > 0$, there exists
a quasihypermetric, non-strictly quasihypermetric space $(Z, d)$
with $|Z| = n$ and $D(Z) = 1$ such that $M(Z) > K$
and with maximal strictly quasihypermetric subspaces of cardinality
$n-1$.
\end{theorem}

\begin{proof}
Write $m = n-4$, so that $m \geq 3$.
Give $X = \{x_1, \dots, x_m\}$ the discrete metric~$d_1$.
Let $Y = \{y_1, y_2, y_3, y_4\}$, where $y_1, y_2, y_3, y_4$
are equally spaced points placed consecutively around a circle
of radius~$2/(\pi m)$, and give~$Y$ the arc-length metric~$d_2$
(see Example~\nwAref{nonstrict} and Corollary~\nwBref{2.4}).
Let $Z = X \cup Y$ and set $c = \frac{1}{2} + \epsilon$
for any~$\epsilon$ satisfying $0 < \epsilon \leq \frac{1}{2}$.
Defining $d \colon Z \times Z \to \R$ as in Theorem~\nwBref{2.1},
we find that $(Z, d)$ is quasihypermetric and non-strictly
quasihypermetric and that $D(Z) = 1$.
Application of Theorem~\nwBref{2.1.4} now gives an invariant
measure of mass~$1$ on~$Z$ with value
\[
\frac{1}{8\epsilon} \biggl(\frac{m-2}{m}\biggr)^2
  + \frac{1}{2} + \frac{\epsilon}{2}.
\]
A second application of Theorem~\nwBref{2.1} shows that
the subspace of~$Z$ obtained by omitting one point of~$Y$
is strictly quasihypermetric, and the result follows.
\end{proof}

\begin{remark}
\label{Knr25}
Although by Corollary~\ref{RW4} all maximal strictly
quasihypermetric subspaces of a given finite quasihypermetric space
have the same cardinality, the proof of the last theorem
allows us to see easily that not all subspaces
of that cardinality need be strictly quasihypermetric.
Indeed, if we remove from~$Z$ in Theorem~\ref{Knr20.2}
a point of~$X$ instead of a point of~$Y$, then the new subspace
has $n-1$ elements but fails to be strictly quasihypermetric
since it has the non-strictly quasihypermetric space~$Y$ as a subspace. 
\end{remark}

\begin{theorem}
\label{Knr30}
$K(n, r) \geq K(n+1, r)$, provided that both numbers are defined.
\end{theorem}

\begin{proof}
Assuming that both the given numbers are defined, we have $n \geq r$
and $n \geq 3$. If $\Xd$ is a space of type $(n+1, r)$,
then every maximal strictly quasihypermetric subspace of~$X$
has $r < n+1$ points. Let $Y = X \setminus \{x\}$,
where $x \in X$ lies outside some fixed maximal strictly
quasihypermetric subspace~$S$ of~$X$. Then $|Y| = n$,
all maximal strictly quasihypermetric subspaces of~$Y$,
of which $S$ is one, have cardinality~$r$ and $0 < D(Y) \leq 1$.
Also, by Theorem~\ref{strict2} and since $S$ is a maximal strictly
quasihypermetric subspace of both $X$ and~$Y$,
we have $M(S) = M(X) = M(Y)$.
Hence, normalizing distances in~$Y$ to give a space~$Y'$
with $D(Y') = 1$, we have $M(Y') \geq M(X)$, and the result follows.
\end{proof}

In addition to the classes of values of $K(n, r)$ given by
Theorems \ref{Kbox} and~\ref{Knr10} above, we note some further
information about individual values of $K(n, r)$,
obtained by more ad hoc arguments.

\begin{theorem}
\label{Knr40}
\mbox{}
\begin{enumerate}
\item
$K(3, 3) = \frac{2}{3}$.
\vspace{4pt}
\item
$K(4, 4) = \frac{3}{4}$.
\vspace{4pt}
\item
$\frac{2}{3} \leq K(5, 4) \leq \frac{3}{4}$.
\vspace{4pt}
\item
$K(6, 4) = \frac{2}{3}$.
\vspace{4pt}
\item
$K(7, 4) = \frac{1}{2}$.
\end{enumerate}
\end{theorem}

\begin{proof}
We omit the proof of~(1), which follows by a straightforward
maximization argument. Theorem~\ref{L1H} below shows that
$K(4, 4) \leq \frac{3}{4}$, and (2) then follows from the observation
that a discrete space~$X$ with $4$ points is strictly quasihypermetric
and has $M(X) = \frac{3}{4}$.
Croft~\cite[pp.\ 305--306]{Croft} gives (without proof) a
classification of the non-obtuse configurations of $6$~points in~$\R^3$.
These fall into three classes. It is easy to check that all configurations
of Croft's types ($\beta$) and~($\gamma$) lie on circumspheres
in~$\R^3$ in such a way that a pair of diametrically opposite points
exists, and it follows from Lemma~\ref{simple} that the corresponding
metric spaces (see the discussion at the beginning of this section),
normalized to have diameter~$1$, have $M$ equal to~$\frac{1}{2}$.
Croft's configurations of type~($\alpha$),
though lying on a circumsphere, do not in general have a pair
of diametrically opposite points, but a direct argument,
which we omit, shows that the normalized values of~$M$
for the corresponding spaces occupy precisely the interval
$(\frac{1}{2}, \frac{2}{3})$. It follows that the values of~$M$ for
spaces of type $(6, 4)$ make up precisely the interval
$[\frac{1}{2}, \frac{2}{3})$, and in particular that
$K(6, 4) = \frac{2}{3}$, proving~(4). (We record explicitly
the fact that the supremum defining $K(6, 4)$ is not attained.)
Croft also remarks that all non-obtuse configurations
of $7$~points in~$\R^3$ consist of $7$~vertices of a parallelepiped,
and since such a configuration clearly has a circumsphere containing
a pair of diametrically opposite points,
Lemma~\ref{simple} gives $K(7, 4) = \frac{1}{2}$, proving~(5).
Finally, (2) and~(4), with Theorem~\ref{Knr30}, give 
$\frac{2}{3} \leq K(5, 4) \leq \frac{3}{4}$, which is~(3).
\end{proof}

\begin{remark}
We do not know the value of $K(6, 5)$, or whether it is finite
or infinite. 
\end{remark}

\section{The Value of $M$ in Finite $L^1$-Embeddable Spaces}

It is well known that every compact $L^1$-embeddable metric space
is quasihypermetric (see, for example, Theorem~\nwCref{2.15});
indeed, such spaces are hypermetric (see Remark~\nwCref{Assouad}).
If we restrict attention within the class of finite quasihypermetric
spaces to the class of finite $L^1$-embeddable spaces,
we can expect more detailed information
about the constant~$M$ to become available.
We derive such information in this section.

We begin with two lemmas.

\begin{lemma}
\label{L1A}
For $n \geq 1$, let $x_1, \dots, x_n \in \R$ and
$\alpha_1, \dots, \alpha_n \in \R$ be such that
$x_1 \leq \dots \leq x_n$ and $\sum_{i=1}^n \alpha_i = 1$.
\vspace{4pt}
\begin{enumerate}
\item[(1)]
If $\alpha_1 < 0$, then
$
\sum_{i,j=1}^n \alpha_i \alpha_j |x_i - x_j|
\leq 
\sum_{i,j=2}^n \beta_i \beta_j |x_i - x_j|,
$
where $\beta_2 = \alpha_1 + \alpha_2$ and
$\beta_3 = \alpha_3, \dots, \beta_n = \alpha_n$.
\vspace{4pt}
\item[(2)]
If $\alpha_n < 0$, then
$
\sum_{i,j=1}^n \alpha_i \alpha_j |x_i - x_j|
\leq 
\sum_{i,j=1}^{n-1} \beta_i \beta_j |x_i - x_j|,
$
where $\beta_1 = \alpha_1, \dots, \beta_{n-2} = \alpha_{n-2}$
and $\beta_{n-1} = \alpha_{n-1} + \alpha_n$.
\end{enumerate}
\end{lemma}

\begin{proof}
Let $\alpha_1 < 0$. Then
\begin{eqnarray*}
\lefteqn{\sum_{i,j=1}^n \alpha_i \alpha_j |x_i - x_j|} \\
 & = &
\sum_{i,j=2}^n \alpha_i \alpha_j |x_i - x_j|
    + 2 \alpha_1 \sum_{j=2}^n \alpha_j x_j
    - 2 \alpha_1 (1 - \alpha_1) x_1 \\
 & \leq &
\sum_{i,j=2}^n \alpha_i \alpha_j |x_i - x_j|
    + 2 \alpha_1 \sum_{j=2}^n \alpha_j x_j
    - 2 \alpha_1 (1 - \alpha_1) x_2 \\
 & = &
\sum_{i,j=2}^n \beta_i \beta_j |x_i - x_j|.
\end{eqnarray*}
The case $\alpha_n < 0$ is similar.
\end{proof}

Note in the lemma that if $\alpha_1 < 0$,
then $\sum_{i=2}^n \beta_i = 1$,
and that if also $\alpha_2 < 0$, then $\beta_2 < 0$.
Similarly, if $\alpha_n < 0$, then $\sum_{i=1}^{n-1} \beta_i = 1$,
and if also $\alpha_{n-1} < 0$, then $\beta_{n-1} < 0$.

\begin{lemma}
\label{L1B}
For $n \geq 1$, let $x_1, \dots, x_n \in \R$ and
$\alpha_1, \dots, \alpha_n \in \R$ be such that
$\sum_{i=1}^n \alpha_i = 1$.
\vspace{4pt}
\begin{enumerate}
\item
$
\sum_{i,j=1}^n \alpha_i \alpha_j |x_i - x_j|
\leq \frac{1}{2} \max_{1 \leq i,j \leq n} |x_i - x_j|.
$
\vspace{4pt}
\item
If $x_1 \leq \dots \leq x_n$ and
$r = \min \{i : 1 \leq i \leq n \mbox{ and } \alpha_i \geq 0\}$
and
$s = \max \{i : 1 \leq i \leq n \mbox{ and } \alpha_i \geq 0\}$,
then
$
\sum_{i,j=1}^n \alpha_i \alpha_j |x_i - x_j| \leq \frac{1}{2} (x_s - x_r).
$
\end{enumerate}
\end{lemma}

\begin{proof}
It is enough to prove~(2). Applying Lemma~\ref{L1A}
step by step, we obtain $\beta_r, \dots, \beta_s \in \R$
with $\sum_{i=r}^s \beta_i = 1$ such that
\[
\sum_{i,j=1}^n \alpha_i \alpha_j |x_i - x_j|
 \leq
\sum_{i,j=r}^s \beta_i \beta_j |x_i - x_j|
 \leq
\frac{1}{2} (x_s - x_r),
\]
by Corollary~\nwBref{2.3}.
\end{proof}

We can now prove our first result on the value of~$M$
in $L^1$-embeddable spaces.

\begin{theorem}
\label{L1C}
Let $X$ be a compact subset of\/ $\R^n$, where $\R^n$ has the metric
induced by the $1$-norm~$\| \cdot \|_1$.
For $1 \leq s \leq n$, let $P_s \colon \R^n \to \R$ be defined by
$
P_s \bigl ((x_1,\ldots ,x_n) \bigr) = x_s
$
for all $(x_1, \ldots, x_n) \in \R^n$. Then
\begin{enumerate}
\item
$\ds M(X) \leq \frac{1}{2} \sum^n_{s=1} D(P_s(X))$ and
\vspace{4pt}
\item
$\ds M(X) \leq \frac{n D(X)}{2}$.
\end{enumerate}
\end{theorem}

\begin{proof}
Suppose that $\alpha_1, \ldots, \alpha_k \in \R$ satisfy
$\ds \sum^k_{i=1} \alpha_i = 1$ and consider $x_1, \ldots, x_k \in X$.
Then
\begin{eqnarray*}
\sum^k_{i,j=1} \alpha_i \alpha_j \| x_i - x_j \|_1
& = &
\sum^n_{s=1} \sum^k_{i,j=1} \alpha_i \alpha_j |P_s(x_i) - P_s(x_j)| \\
& \leq &
\frac{1}{2} \sum^n_{s=1} D(P_s(X)),
\end{eqnarray*}
by Lemma~\ref{L1B}, part~(1).
Since
\[
|P_s(x) - P_s(y)| = |P_s(x-y)| \leq \|x-y\|_1
\]
for all $x,y \in \R^n$, we get
\[
D(P_s(X))\leq D(X)
\]
for $1 \leq s \leq n$ and hence $M(X) \leq n D(X)/2$.
\end{proof}

\begin{remark}
\label{L1D}
We observe that the inequality $M(X) \leq n D(X)/2$ is sharp
in each dimension.

First, for $n = 1$, we have $M(X) = D(X)/2$
for every compact subset~$X$ of~$\R$.
To see this, note that $\mu = (\delta_\alpha + \delta_\beta)/2$ is
$d$-invariant on~$X$, where $\alpha = \min_{x\in X} x$
and $\beta = \max_{x\in X} x$, and since
$d_\mu(x) = (\beta -\alpha)/2 = D(X)/2$ for all $x$ in $X$,
we therefore have $\ds M(X) = D(X)/2$.

Now let $n \geq 2$. Define
$X_n = \{ \pm e_1, \pm e_2, \ldots ,\pm e_n, z \}$,
where $e_1, \ldots, e_n$ are the canonical unit vectors
and $z = (0, 0, \ldots ,0 )$. Let
\[
\mu_n
 =
\frac{1}{2} \sum_{i=1}^n (\delta_{e_i} + \delta_{-e_i}) - (n-1)\delta_z.
\]
Then $\mu_n \in \Mone(X_n)$ and $d_{\mu_n}(x) = n$ for all $x \in X_n$,
and hence $M(X_n) = n = n D(X_n)/2$.
\end{remark}

Theorem~\ref{L1G2} below will show that the example above is minimal
for each $n \geq 2$; that is, we will show that
$M(X) < n D(X)/2$ for all finite sets~$X$
in $(\R^n, \| \cdot \|_1$) with $|X| < |X_n| = 2n+1$.

We need first the following simple lemmas.

\begin{lemma}
\label{L1E}
Let $x_1, \ldots, x_k \in \R$. Then
\[
\sum_{1\leq i<j\leq k} |x_i - x_j|
\geq
(k-1) \max_{1\leq i,j\leq k} |x_i - x_j|.
\]
\end{lemma}

\begin{proof}
Without loss of generality, suppose that
$\max_{1\leq i,j\leq k} |x_i - x_j| = |x_1 - x_2|$.
Then
\begin{eqnarray*}
\sum_{1\leq i<j\leq k} |x_i - x_j|
& \ge &
|x_1 - x_2| + \sum^k_{i=3} \bigl( |x_1 - x_i|+ |x_2 - x_i| \bigr) \\
& \geq & (k-1) |x_1 - x_2|.
\end{eqnarray*}
\end{proof}

\begin{lemma}
\label{L1F}
Let $X = \{ x_1, \ldots, x_k\} \subseteq \R^n$,
where $\R^n$ has the metric induced by the $1$-norm~$\|\cdot\|_1$.
Then
\[
\sum_{s=1}^n D(P_s(X)) \leq (k/2) D(X).
\]
\end{lemma}

\begin{proof}
Applying Lemma~\ref{L1E} for the last step, we have
\begin{eqnarray*}
D(X) \binom{k}{2}
& \geq &
\sum_{1\leq i<j\leq k} \| x_i - x_j \|_1 \\
& = &
\sum_{s=1}^n \sum_{1\leq i<j\leq k} |P_s(x_i) - P_s(x_j)| \\
& \geq &
(k-1) \sum_{s=1}^n D(P_s(X)),
\end{eqnarray*}
giving the result.
\end{proof}

We can now prove our final results on $L^1$-embeddable spaces.
First, by combining Theorem~\ref{L1C}, part~(1) and Lemma~\ref{L1F},
we have the following result immediately.

\begin{theorem}
\label{L1G1}
Let $X$ with $|X| = k$ be a finite subset of\/~$\R^n$,
where $\R^n$ has the metric induced by the $1$-norm~$\|\cdot\|_1$.
Then
\[
M(X) \leq \frac{k}{4} \cdot D(X).
\]
\end{theorem}

The second result refines the first in the case when
there is an appropriate bound on the cardinality of the subset.

\begin{theorem}
\label{L1G2}
Let $n \geq 2$ and let $X$ with $|X| = k \leq 2n$ be a finite subset
of\/~$\R^n$, where $\R^n$ has the metric induced
by the $1$-norm~$\|\cdot\|_1$. Then
\[
M(X)
\leq
\min \left( \frac{k}{4}, \,\frac{n}{2} - \frac{1}{4} \right) \cdot D(X).
\]
\end{theorem}

\begin{proof}
In view of the previous theorem, it remains only to show that
$M(X) \leq \bigl( (2n-1)/4 \bigr) D(X)$ for $|X| = 2n$.
Let $X = \{ x_1, x_2, \ldots, x_{2n} \}$.
For each $s$ such that $1 \leq s \leq n$, choose $a_s, b_s \in X$
such that
\[
D(P_s(X)) = |P_s(a_s) - P_s(b_s)|.
\]
Let $Y = \{a_1, b_1, \ldots , a_n, b_n \}$.
If $|Y| < 2n$, then Theorem~\ref{L1C}, part~(1) implies that
\begin{eqnarray*}
M(X)
& \leq &
\frac{1}{2} \sum_{s=1}^n D(P_s(X)) \\
& = & \frac{1}{2} \sum_{s=1}^n D(P_s(Y)) \\
& \leq &
\frac{|Y|}{4} D(Y) \\
& \leq &
\frac{2n-1}{4} D(X),
\end{eqnarray*}
by Lemma~\ref{L1F}.
Thus, let us assume that $|Y| = 2n$ and hence that
$X = \{ a_1, b_1, \ldots, a_n, b_n \}$.
Now let $\alpha_1, \alpha_2, \ldots, \alpha_{2n} \in \R$ such that
$\sum_{i=1}^{2n} \alpha_i = 1$. If $\alpha_i \geq 0$ for all
$1 \leq i\leq 2n$, we obtain
\begin{eqnarray*}
\sum^{2n}_{i,j=1} \alpha_i \alpha_j \|x_i - x_j\|_1
& \leq &
D(X) \sum_{i=1}^{2n} \alpha_i (1 - \alpha_i) \\
& = &
D(X) \biggl( 1 - \sum_{i=1}^{2n} \alpha_i^2 \biggr) \\
& \leq &
D(X) \Bigl( 1 - \frac{1}{2n} \Bigr)\\
& \leq &
\frac{2n-1}{4} D(X).
\end{eqnarray*}
If without loss of generality $\alpha_1 < 0$, let us choose~$i_0$
with $1 \leq i_0 \leq n$ such that $x_1 \in \{ a_{i_0}, b_{i_0} \}$,
noting that the choice is unique because
$X = \{ a_1, b_1, \ldots , a_n, b_n \}$.

Since
\[
\min_{y\in P_{i_0}(X)} y = P_{i_0}(x_1)
\quad\mbox{ or }\quad
\max_{y\in P_{i_0}(X)} y = P_{i_0}(x_1),
\]
we conclude that
\[
\sum^{2n}_{i,j=1} \alpha_i \alpha_j |P_{i_0}(x_i) - P_{i_0}(x_j)|
\leq
\frac{1}{2} D(P_{i_0}(X'))
\]
by Lemma~\ref{L1B}, part~(2), where $X' = X \setminus \{x_1\}$.
Since $D(P_s(X)) = D(P_s(X'))$ for all $s \neq i_0$,
we obtain
\begin{eqnarray*}
\sum^{2n}_{i,j=1} \alpha_i \alpha_j \|x_i - x_j\|_1
& = &
\sum_{s=1}^{n} \sum_{i,j=1}^{2n}
    \alpha_i \alpha_j |P_s(x_i) - P_s(x_j)| \\
& \leq &
\frac{1}{2} \sum_{s=1}^{n} D(P_s(X')) \\
& \leq &
\frac{2n-1}{4} D(X') \\
& \leq &
\frac{2n-1}{4} D(X),
\end{eqnarray*}
by Lemma~\ref{L1B}, part~(1) and Lemma~\ref{L1F}.
\end{proof}

If $\Xd$ is a metric space with at most $4$ points,
then by~\cite{Wolfe}, $\Xd$ can be isometrically embedded in
$(\R^2, \| \cdot \|_\infty)$, and it is then immediate that
it can also be isometrically embedded in $(\R^2, \| \cdot \|_1)$.
As a corollary of the previous theorem, we therefore have the following.

\begin{theorem}
\label{L1H}
Let $\Xd$ be a finite metric space of at most four points.
Then $M(X) \leq \frac{3}{4} D(X)$.
\end{theorem}

\begin{remark}
\label{L1I}
As noted above, $4$-point spaces are $L^1$-embeddable.
On the other hand, it is well known that $5$-point spaces
need not be $L^1$-embeddable.
For example, Assouad \cite[Proposition~2]{Ass2} constructs
a $5$-point space which is quasihypermetric but not hypermetric,
and is therefore not $L^1$-embeddable.
We noted in Example~\nwCref{Assouad1} that Assouad's space
has $M$ infinite.

A different class of spaces, with $M$ finite, is provided
by the work of the present paper, since Theorem~\ref{L1G1},
for example, can be applied to give a necessary condition
for $L^1$-embeddability. Specifically, Theorem~\ref{Knr20.1}
provides a class of quasihypermetric spaces of cardinality~$5$,
diameter~$1$ and with $M$ finite but arbitrarily large,
while by Theorem~\ref{L1G1} such a space can be $L^1$-embeddable
only if $M$ has value at most $5/4$.
\end{remark}

\bibliographystyle{amsplain}
\bibliography{bibliography}

\end{document}